\documentclass[10pt]{amsart}

\title{Metric foliations of homogeneous three-spheres}

\author{Meera Mainkar}
\address{Dept. of Mathematics\\
Central Michigan University\\
Mount Pleasant, MI 48859}

\author{Benjamin Schmidt}
\address{Dept. of Mathematics \\ 
                 Michigan State University \\
                 East Lansing, MI 48824}

\thanks{ The authors appreciate the anonymous referee's careful reading of a preliminary version and suggestions for improvement.  The first author thanks Michigan State University for its hospitality while this research was conducted.}

\date{\today}

\newtheorem{thm}{Theorem}[section]
\newtheorem{lem}[thm]{Lemma}
\newtheorem{cor}[thm]{Corollary}
\newtheorem{prop}[thm]{Proposition}

\theoremstyle{definition}
\newtheorem{rem}{Remark}[section]

\newtheorem{defn}{Definition}[section]

\numberwithin{equation}{section}

\newcommand{\G}{\Gamma}

\newcommand{\n}{\nabla}

\def\Pb{\ifmmode{\Bbb P}\else{$\Bbb P$}\fi}
\def\Z{\ifmmode{\Bbb Z}\else{$\Bbb Z$}\fi}
\def\Q{\ifmmode{\Bbb Q}\else{$\Bbb Q$}\fi}
\def\C{\ifmmode{\Bbb C}\else{$\Bbb C$}\fi}
\def\R{\ifmmode{\Bbb R}\else{$\Bbb R$}\fi}
\def\H{\ifmmode{\Bbb H}\else{$\Bbb H$}\fi}
\def\S{\ifmmode{S^2}\else{$S^2$}\fi}

\def\tr{\operatorname{trace}}
\def\det{\operatorname{det}}

\def\S{\mathcal S}
\def\G{\operatorname{SU(2)}}
\def\lie{\operatorname{su(2)}}

\def\Exp{\operatorname{Exp}}

\usepackage{graphicx} 

\usepackage{amsmath,amsthm,amsfonts,amscd,flafter,epsf}
\usepackage{graphics}
\usepackage{epsfig}
\usepackage{psfrag}

\begin{document}

\begin{abstract}
A smooth foliation of a Riemannian manifold is \textit{metric} when its leaves are locally equidistant and is \textit{homogeneous} when its leaves are locally orbits of a Lie group acting by isometries.  Homogeneous foliations are metric foliations, but metric foliations need not be homogeneous foliations.  

We prove that a homogeneous three-sphere is naturally reductive if and only if all of its metric foliations are homogeneous.
\end{abstract}
\maketitle

\setcounter{secnumdepth}{1}

\setcounter{section}{0}

\section{\bf Introduction}
A smooth foliation of a Riemannian manifold is a \textit{metric foliation} when its leaves are locally equidistant.\footnote{Metric foliations are also referred to as \textit{Riemannian foliations} in the literature.} For example, the fibers of a Riemannian submersion are locally equidistant and so define a metric foliation of the total space.  A smooth foliation (or submersion) is \textit{homogeneous} when, locally, its leaves (or fibers) are orbits of an isometric group action.  If a foliation or submersion is homogeneous, then it is also metric.  It is an interesting problem to determine to what extent the converse holds on a Riemannian manifold with a large isometry group. 

Simply connected constant curvature spaces have isometry groups of largest possible dimension.  Metric foliations of curvature one spheres are either homogeneous or metrically congruent to the Hopf fibration $S^{15} \rightarrow S^{8}(\frac{1}{2})$ \cite{GrGr, LyWi}.  Metric foliations of Euclidean space all arise from Riemannian submersions \cite{FGLT} and are homogeneous if the submersion fibers are connected \cite{GrWa1, GrWa2, SpWe}.  One-dimensional metric foliations of hyperbolic spaces are classified \cite{LeYi} and are mostly inhomogeneous.

The following are a few known results concerning the above problem for Riemannian manifolds with variable curvatures: one-dimensional metric foliations of compact Lie groups, equipped with a bi-invariant metric, are homogeneous \cite{Mu1}; many compact Lie groups, equipped with a bi-invariant metric, are the total space of an inhomogeneous Riemannian submersion \cite{KeSh};  one-dimensional and codimension one metric foliations of the Heisenberg groups $H_n$, equipped with a left-invariant metric, are homogeneous \cite{Mu2, Wa}; Riemannian submersions from $M^3=S^2 \times \mathbb{R}$, equipped with a product Riemannian metric, to a surface are homogeneous \cite{GrTa}.\\

\noindent \textbf{Main Theorem.} \textit{A Riemannian homogeneous three-sphere is naturally reductive if and only if all of its metric foliations are homogeneous.\\}

To prove the Main Theorem, one only has to consider metric foliations with one-dimensional leaves:  a codimension one smooth foliation of the three-sphere has some noncompact leaves \cite{No} and therefore is not a metric foliation with respect to any Riemannian metric \cite{Gh}.

The naturally reductive homogeneous three-spheres are homothetic to a Berger sphere or isometric to a constant curvature sphere \cite{TrVa}.  As a consequence, in order to prove the Main Theorem, it suffices to prove the following two theorems. 

\begin{thm}\label{main1}
A homogeneous three-sphere that is not naturally reductive admits a one-dimensional metric foliation that is not homogeneous.
\end{thm}

\begin{thm}\label{main2}
All one-dimensional metric foliations of a three-dimensional Berger sphere are homogeneous.
\end{thm}

As the three-sphere is closed and simply connected, one-dimensional homogeneous foliations are orbit foliations of \textit{globally} defined isometric flows (see e.g. Corollary \ref{closed} below).  We therefore have the following Corollary of Theorem \ref{main2}.

\begin{cor}
All one-dimensional metric foliations of a three-dimensional Berger sphere are orbit foliations of a globally defined isometric flow.
\end{cor}

The paper is organized as follows.  Preliminary material about naturally reductive spaces, homogeneous three-spheres, and one-dimensional metric foliations is summarized in Section 2.  Section 3 discusses Berger spheres as both left-invariant metrics on $\G$ and as naturally reductive spaces.  Theorem \ref{main1} is proved in Section 4 and Theorem \ref{main2} is proved in Section 5.

\section{\bf Preliminaries.\\}

\subsection{Naturally reductive spaces.\\}
In this section, we quickly review naturally reductive spaces.  The Berger spheres, discussed in Section \ref{berger}, are examples of naturally reductive spaces.

Throughout, a \textit{coset space} refers to a smooth manifold $M=G/H$ where $G$ is a Lie group with Lie algebra $\mathfrak{g}$ and $H$ is a closed subgroup of $G$ with Lie subalgebra $\mathfrak{h}$.  Furthermore, the Lie groups $G$ and $H$ are both assumed to be connected, an additional assumption suitable for our purposes. 

Let $\pi:G \rightarrow M$, $g \mapsto gH$, denote the quotient map, $e \in G$ the identity element, and $o=\pi(e)=H \in M$. The map $\pi$ is equivariant with respect to the natural $G$ actions on $G$ and $M$ by left-translations.

For $g \in G$, let $C_g:G \rightarrow G$ and $L_g:M \rightarrow M$ denote conjugation and left-translation by $g$, respectively.  If $h \in H$, then $L_h(o)=o$ and $\pi \circ C_h= L_h \circ \pi.$ Differentiation of this equality at $e \in G$ implies that for each $v \in \mathfrak{g}$ and $h \in H$,\begin{equation}\label{obv2}
d\pi_e(Ad(h)v)=dL_h(d\pi_e(v)).
\end{equation}

\begin{defn}\label{red}
A \textit{reductive decomposition} for the coset space $M=G/H$ is  a vector subspace $\mathfrak{m} \subset \mathfrak{g}$  satisfying 
\begin{enumerate}
\item $\mathfrak{g}=\mathfrak{m} \oplus \mathfrak{h}$, and 
\item $Ad(H)\mathfrak{m} \subset \mathfrak{m}$. 
\end{enumerate} A coset space $M=G/H$ is \textit{reductive} if it admits a reductive decomposition.
\end{defn}

\begin{rem}\label{mu}
By (1),  $d\pi_{e}\vert_{\mathfrak{m}}:\mathfrak{m} \rightarrow T_oM$ is a linear isomorphism.  We let $\mu=d\pi_{e}\vert_{\mathfrak{m}}$ denote this isomorphism. As $H$ is assumed connected, (2) is equivalent to $ad(\mathfrak{h})\mathfrak{m} \subset \mathfrak{m}.$
\end{rem}

If $\mathfrak{m} \subset \mathfrak{g}$ is a reductive decomposition as above, then the isomorphism $\mu$ induces a bijective correspondence between $Ad(H)$-invariant inner products on $\mathfrak{m}$ and $G$-invariant Riemannian metrics on $M$.  This well known fact is derived using equation (\ref{obv2}).  

\begin{defn}\label{natred}
Let $M=G/H$ be a reductive coset space with reductive decomposition $\mathfrak{m} \subset \mathfrak{g}$.  The reductive decomposition $\mathfrak{m}$ is \textit{naturally reductive} if there exists an $Ad(H)$-invariant inner product $\langle \cdot, \cdot \rangle$ on $\mathfrak{m}$ with the additional property that for each $x,y,z \in \mathfrak{m}$, $$\langle [x,y]_{\mathfrak{m}},z\rangle+\langle [x,z]_{\mathfrak{m}},y\rangle=0.$$ 
A coset space $M=G/H$ is \textit{naturally reductive} if it admits a naturally reductive decomposition $(\mathfrak{m},\langle \cdot, \cdot \rangle)$.  A Riemannian manifold is \textit{naturally reductive} if it is isometric to a naturally reductive coset space.
\end{defn}

\begin{rem}\label{homgeo}
Let $M=G/H$ be a naturally reductive homogeneous space with naturally reductive decomposition $(\mathfrak{m},\langle \cdot,\cdot \rangle)$ and let $\exp: \mathfrak{g} \rightarrow G$ denote the Lie exponential map. If $c(t)$ is a geodesic in $M$ with $c(0)=o$ and $c'(0)=v$, then $$c(t)=\pi\circ\exp(t \mu^{-1}(v)),$$ where $\mu$ is defined as in Remark \ref{mu}. This fact will be used to analyze the behavior of geodesics in Berger spheres. 
\end{rem}

\subsection{Riemannian homogeneous three-spheres.\\}

A homogeneous three-sphere is isometric to a left-invariant metric on $\G$ \cite{Se}.  By \cite{Mi}, for each left-invariant metric $g$ on $\G$, there are positive real numbers $x,y,z \in \mathbb{R}$, and a $g$-orthonormal left-invariant framing $\{E_1,E_2,E_3\}$ of $\G$  with structure constants  \begin{equation}\label{structure2}
[E_1,E_2]=2x E_3\,\,\,\,\,\, [E_2,E_3]=2y E_1\,\,\,\,\,\, [E_3,E_1]=2z E_3.
\end{equation}

\begin{rem}\label{isometry}
The symmetric group $Sym(3)$ acts on $\mathbb{R}^3$ by permuting coordinates.  The $Sym(3)$-orbit of $(x,y,z)$ determines the isometry class of the left-invariant metric $g$ \cite{BFSTW}.  
\end{rem}
An isometry class is naturally reductive if and only if at most two of the structure constants are distinct \cite{TrVa}. Those with three equal structure constants have constant sectional curvatures.  Those with two distinct structure constants are homothetic to a Berger sphere as will be described in Section \ref{berger}.  The remaining isometry classes have three distinct structure constants.

Given a left-invariant metric $g$ and orthonormal framing as in (\ref{structure2}), let $\Gamma_{ij}^k=g(\n_{E_i} E_j,E_k)$ denote the Christoffel symbols.  Use Koszul's formula to compute \begin{eqnarray}\label{use}\nonumber
\Gamma_{12}^3=(x+z-y)=-\Gamma_{13}^2,\\ 
\Gamma_{23}^1=(x+y-z)=-\Gamma_{21}^3,\\ \nonumber
\Gamma_{31}^2=(y+z-x)=-\Gamma_{32}^1.
\end{eqnarray}
The remaining Christoffel symbols are zero.  


\subsection{One-dimensional metric foliations.\\}

Let $(M,g)$ be a Riemannian manifold and $\mathcal{F}$ a smooth foliation of $M$.  Recall that $\mathcal{F}$ is \textit{metric} if its leaves are locally equidistant and is \textit{homogeneous} if locally, its leaves are orbits of an isometric group action.

Equivalently, the foliation $\mathcal{F}$ is metric if its orthogonal distribution, $T\mathcal{F}^{\perp}$, is a \textit{totally geodesic} distribution.  In terms of vector fields, $T\mathcal{F}^{\perp}$ is totally geodesic provided that whenever $X,Y$ are sections of $T\mathcal{F}^{\perp}$, then so too is $\n_X Y+\n_Y X$.

We restrict our attention to the case of one-dimensional foliations.  Locally, the foliation $\mathcal{F}$ is oriented with a unit length vector field $V$ tangent to $\mathcal{F}$.  

\begin{defn}
The mean curvature form of a smooth one-dimensional foliation $\mathcal{F}$ is the one form $\omega \in \Omega^1(M)$ defined by $\omega(\cdot)=g(\n_V V,\cdot).$ 
\end{defn}

The mean curvature form does not depend on the choice of a local orientation of $\mathcal{F}$.

\begin{lem}\label{exact}
A smooth one-dimensional foliation of a Riemannian manifold is homogeneous if and only if the mean curvature form of the foliation is closed.
\end{lem}

\begin{proof}
Let $\mathcal{F}$ be a smooth one-dimensional foliation of a Riemannian manifold $M$ and $\omega$ its mean curvature form.  By Poincare's Lemma, $\omega$ is closed if and only if $\omega$ is locally exact.  The one form $\omega$ is locally exact if and only if each point admits a neighborhood $B$ and a smooth function $f:B \rightarrow \mathbb{R}$ satisfying $df=\omega$ on $B$.  A straightforward computation shows that if $V$ is a unit length vector field orienting $\mathcal{F}$ on $B$, then $df=\omega$ on $B$ if and only if the vector field $X=e^{-f} V$ is Killing on $B$, concluding the proof.
\end{proof}

\begin{cor}\label{closed}
If $M$ is a closed and simply connected Riemannian manifold, then a one-dimensional smooth foliation $\mathcal{F}$ of $M$ is homogeneous if and only if $\mathcal{F}$ is the orbit foliation of a globally defined isometric flow on $M$.
\end{cor}

\begin{proof}
As $M$ is simply connected, $\mathcal{F}$ is globally oriented with a unit vector field $V$.  The Corollary follows from the proof of Lemma \ref{exact} since $H_{dR}^{1}(M)=0$.  
\end{proof}

\section{Berger spheres.}\label{berger}
This section reviews and relates different constructions of the Berger spheres with the eventual goal of describing geodesics in a left-invariant model.

\subsection{Berger spheres as left-invariant metrics on SU(2).\\}

Let $S^{3}=\{(z,w) \in \mathbb{C}^2\, \vert\, z\bar{z}+w\bar{w}=1\}$ and $S^1=\{e^{i\theta} \in \mathbb{C}\, \vert \, \theta \in \mathbb{R}\}$.  Let $g_1$ denote the canonical Riemannian metric on $S^3$ induced from the Euclidean inner-product on $\mathbb{R}^4.$  The action $S^1 \times S^3 \rightarrow S^3$ defined by $e^{i\theta} \cdot (z,w) = (e^{i\theta} z, e^{i\theta} w)$ is $g_1$-isometric.  Let $X$ denote the unit-length $g_1$-Killing field generating this action: $X_{(z,w)}=(iz,iw).$  Let $\mathcal{V}$ denote the line field spanned by $X$ and $\mathcal{H}$ the $g_1$-orthogonal distribution to $\mathcal{V}$.

 \begin{defn}[Berger Sphere I]\label{bergerd}
 For $\epsilon>0$, define a Riemannian metric $g_\epsilon$ on $S^3$ as follows: The distributions $\mathcal{V}$ and $\mathcal{H}$ are $g_\epsilon$-orthogonal, $g_\epsilon=\epsilon g_1$ on $\mathcal{H}$, and $g_{\epsilon}=\epsilon^2 g_1$ on $\mathcal{V}$.  A \textit{Berger sphere} is a Riemannian manifold of the form $(S^3,g_\epsilon)$ for some $\epsilon>0$ and $\epsilon \neq 1$.
 \end{defn}
 
 \begin{rem}
 Typically, Berger spheres are defined by solely rescaling vectors tangent to $\mathcal{V}$.  The description in Definition \ref{bergerd} differs from this alternative rescaling by a homothety and is more suitable for our purposes.
 \end{rem}
 
 \begin{rem}\label{iso}
 It is immediate from the description of $g_\epsilon$ that the above circle action is $g_\epsilon$-isometric, or equivalently, that the vector field $X$ is $g_{\epsilon}$-Killing.
  \end{rem}
 
Let $M_2(\mathbb{C})$ denote the set of $2 \times 2$ complex matrices.  For $A \in M_2(\mathbb{C})$, let $A^{*}$ denote the conjugate transpose of $A$ and let $e \in M_2(\mathbb{C})$ denote the identity matrix.  The Lie group $\G=\{A \in M_2(\mathbb{C})\, \vert\, AA^*=A^*A=e,\,\,\, \det(A)=1\}$  has Lie algebra $\lie=T_e\G=\{ A \in M_2(\mathbb{C})\, \vert\, A^*=-A, \,\,\,\, \tr(A)=0\}$. The matrices $$x_1=\begin{pmatrix} 0 & 1\\ -1 & 0  \end{pmatrix},\,\,\,\,\, x_2=\begin{pmatrix} 0 & i\\ i & 0  \end{pmatrix}, \,\,\,\,\, x_3=\begin{pmatrix} i & 0\\ 0 & -i  \end{pmatrix}$$ form a basis of $\lie$ with structure constants  \begin{equation}\label{strut1}
[x_1,x_2]=2x_3\,\,\,\,\,\, [x_2,x_3]=2x_1\,\,\,\,\,\, [x_3,x_1]=2x_2.
\end{equation}

For $i=1,2,3,$ let $X_i$ denote the left-invariant vector field on $\G$ with $X_i(e)=x_i$.  Let $h_1$ be the orthonormalizing metric for the framing $\{X_1,X_2,X_3\}$.  Then $h_1$ is a bi-invariant metric on $\G$.  The map $F:S^3 \rightarrow \G$ defined by $$F((z,w))=\begin{pmatrix} z & -\bar{w}\\ w & \bar{z} \end{pmatrix}$$ is an isometry between $(S^3,g_1)$ and $(\G,h_1)$ with $dF(X)=X_3$ and with $X_1,X_2$ tangent to $dF(\mathcal{H})$. 

\begin{defn}[Berger Sphere II]\label{berger2}
For $\epsilon>0$, let $h_\epsilon$ be the left-invariant metric on $\G$ for which the $X_i$ are $h_\epsilon$-orthogonal, $h_\epsilon(X_3,X_3)=\epsilon^2$, and $h_{\epsilon}(X_1,X_1)=h_{\epsilon}(X_2,X_2)=\epsilon.$
\end{defn}

For each $\epsilon>0$, the map $F:(S^{3},g_\epsilon) \rightarrow (\G,h_\epsilon)$ is an isometry.  By Remark \ref{iso}, the left-invariant vector field $X_3$ is a Killing field for $(\G,h_\epsilon)$, a fact used in the next subsection.   Setting $$Y_1=\epsilon^{-1/2} X_1,\,\,\,\,\, Y_2=\epsilon^{-1/2} X_2,\,\,\,\,\, Y_3=\epsilon^{-1} X_3,$$ the left-invariant vector fields $\{Y_i\}$ constitute an $h_{\epsilon}$-orthonormal framing of $\G$ with structure constants  \begin{equation}\label{strut2}
[Y_1,Y_2]=2Y_3,\,\,\,\,\,\, [Y_2,Y_3]=2\epsilon^{-1}Y_1,\,\,\,\,\,\, [Y_3,Y_1]=2\epsilon^{-1}Y_2.
\end{equation}

\begin{rem}\label{homoth}
As mentioned in Section 2, a naturally reductive homogeneous three-sphere that is not of constant sectional curvatures is isometric to a left-invariant metric on $\G$ which is homothetic to a metric $h_{\epsilon}$ as described above.
\end{rem}

\subsection{Berger spheres as naturally reductive spaces.\\} 
In this subsection, the Berger sphere $(SU(2),h_{\epsilon})$ is shown to be a naturally reductive Riemannian manifold. This fact is used to describe a property of its geodesics in the concluding Proposition \ref{GEO}.

Let $\exp:\lie \rightarrow \G$ denote the Lie exponential map. The left-invariant vector field $X_3$ generates an $h_{\epsilon}$-isometric flow $\Phi^{s}: \G \rightarrow \G$, $s \in \mathbb{R}$, with orbit through $g \in \G$ given by \begin{equation}\label{flow} \Phi^{s}(g)=g\exp(sx_3)=g\begin{pmatrix} e^{is} & 0\\ 0 & e^{-is}  \end{pmatrix}.\end{equation} Let $G=\mathbb{R} \times \G$.  The transitive action $G \times \G \rightarrow \G$ defined by $$((s,g),\bar{g}) \mapsto (s,g)\cdot\bar{g}:= g\bar{g}\begin{pmatrix} e^{-is} & 0\\ 0 & e^{is} \end{pmatrix}$$ is by $h_{\epsilon}$-isometries.  The isotropy group of $e\in \G$ for this action is $$H=\{(s,g) \in G\, \vert\, (s,g) \cdot e=e\}=\{(l, \begin{pmatrix} e^{il} & 0\\ 0 & e^{-il}  \end{pmatrix})\, \vert\, l\in\mathbb{R}\}.$$  The connected group $H$ is closed in $G$.  Let $\Theta:G/H \rightarrow \G$ denote the $G$-equivariant diffeomorphism defined by $$\Theta((s,g)H)=(s,g)\cdot e=g\begin{pmatrix} e^{-is} & 0\\ 0 & e^{is} \end{pmatrix}.$$ Then $\Theta(o)=e$ and the pullback metric $\Theta^{*}h_\epsilon$ is $G$-invariant.  The $G$-invariant metric $\Theta^{*}h_\epsilon$ on the coset space $G/H$ arises from a naturally reductive decomposition $(\mathfrak{m},\langle \cdot, \cdot \rangle)$ that we now describe.

The Lie algebra $\mathfrak{g}=\mathbb{R} \oplus \lie$ of $G$ admits the following basis:
$$b_0=(1,0),\,\,\,\,\,b_1=(0,x_1),\,\,\,\,\,b_2=(0,x_2),\,\,\,\,\,b_3=(0,x_3).$$ Let $\Exp: \mathfrak{g} \rightarrow G$ denote the Lie exponential map and let $u=b_o+b_3$. Then $H=\{ \Exp(lu)\, \vert\, l\in \mathbb{R}\}$  and $\mathfrak{h}=\langle \{u\} \rangle$ is the Lie subalgebra of $\mathfrak{g}$ corresponding to $H$.  Let $v=(\epsilon-1)b_0+\epsilon b_3$ and define $\mathfrak{m}=\langle \{v,b_1,b_2\} \rangle \subset \mathfrak{g}$.  Then $\mathfrak{g}=\mathfrak{m} \oplus \mathfrak{h}$.  Use (\ref{strut1}) to calculate \begin{equation}\label{strut3}
[u,v]=0,\,\,\,\,\,\, [u,b_1]=2b_2,\,\,\,\,\,\, [u,b_2]=-2b_1.
\end{equation}  Conclude that $ad(\mathfrak{h})\mathfrak{m} \subset \mathfrak{m}$ and that $\mathfrak{m}$ is a reductive decomposition for $M=G/H$. 

The $G$-invariant metric $\Theta^{*}h_{\epsilon}$ induces an $Ad(H)$-invariant inner product $\langle \cdot,\cdot \rangle$ on $\mathfrak{m}$.  To determine this inner product, evaluate $$d\Theta_o(\mu(b_1))=x_1,\,\,\,\,\,d\Theta_o(\mu(b_2))=x_2,\,\,\,\,\,d\Theta_o(\mu(v))=x_3,$$ to conclude that $\{b_1,b_2,v\}$ are $\langle \cdot, \cdot \rangle$-orthogonal, that $\langle b_i,b_i\rangle=\epsilon$ when $i=1,2$, and that $\langle v,v\rangle=\epsilon^2$. Calculate \begin{equation}\label{ff}
[v,b_1]_{\mathfrak{m}}=2\epsilon b_2, \,\,\,\, [v,b_2]_{\mathfrak{m}}=-2\epsilon b_1, \,\,\,\, [b_1,b_2]_{\mathfrak{m}}=2v.
\end{equation}  Using the above description of $\langle \cdot,\cdot\rangle$ and (\ref{ff}), it is straightforward to verify that $(\mathfrak{m}, \langle \cdot, \cdot \rangle)$ is a naturally reductive decomposition for the coset space $G/H$.  This naturally reductive decomposition induces the $G$-invariant metric $\Theta^{*}h_{\epsilon}$ by construction.

\begin{rem}\label{rot}
Let $e_1=\epsilon^{-1/2}b_1$, $e_2=\epsilon^{-1/2}b_2$, and $e_3=\epsilon^{-1}v$.  The above analysis shows that $\{e_1,e_2, e_3\}$ is an orthonormal basis of $(\mathfrak{m}, \langle \cdot,\cdot \rangle)$.  This orthonormal basis is carried to the $h_\epsilon$-orthonormal basis $\{Y_1(e),Y_2(e),Y_3(e)\}$ of $\lie$ under the linear isometry $d\Theta_0\circ d\pi_e$.  It follows from (\ref{obv2}) and (\ref{strut3}) that the isotropy action of $H$ on $(T_{e}\G,h_\epsilon)$ is by rotations about the axis spanned by $Y_3(e)$.  
\end{rem}

We conclude this subsection with a proposition about geodesics in the Berger sphere $(\G,h_\epsilon)$.  To this end, let $\theta \in (0,\pi)$ and define

 \begin{eqnarray}\label{ST} \alpha=\alpha_{\theta}=\cos(\theta), &  \beta=\beta_{\theta}=\sin(\theta), & m=m_{\theta}=\sqrt{\alpha^2+\epsilon^{-1}\beta^2},  \end{eqnarray} 
  \begin{eqnarray} T =T_\theta=2\pi m^{-1},  & S=S_\theta=\alpha(1-\epsilon)T.   \nonumber \end{eqnarray} 
 Moreover, recall that the left-invariant vector field $Y_3$ on $\G$ is $h_{\epsilon}$-Killing and of unit length.  In particular, its orbits are $h_{\epsilon}$-geodesics.  Let $$\phi^{s}:(\G,h_\epsilon) \rightarrow (\G,h_{\epsilon}),\,\,\, s\in \mathbb{R},$$ be the isometric flow generated by $Y_3$.  As $Y_3=\epsilon^{-1}X_3$, recalling that $\Phi^{s}$ denotes the flow generated by $X_3$, we have that $\Phi^{s/\epsilon}=\phi^{s}$.  Comparing with (\ref{flow}), the orbit of $\phi^{s}$ through $g\in \G$ is given by $$\phi^{s}(g)=g\begin{pmatrix} e^{i\epsilon^{-1}s} & 0\\ 0 & e^{-i\epsilon^{-1}s}  \end{pmatrix}.$$  

\begin{prop}\label{GEO}
Let $c: \mathbb{R} \rightarrow (\G,h_{\epsilon})$ be a unit speed geodesic.  If $c'(0)$ and $Y_3(c(0))$ make angle $\theta \in (0,\pi)$, then $$c(t+T_{\theta})=\phi^{S_{\theta}}(c(t)).$$ 
\end{prop}

\begin{proof}
Let $\alpha, \beta, m, T, S$ be as defined in (\ref{ST}).  By applying isometries from $G=\mathbb R \times \G$, we may assume that $c(0)=e$ and that $c'(0)=\alpha Y_3(e)+\beta Y_2(e)$.  Under the inverse of $d\Theta_o \circ \mu:\mathfrak{m} \rightarrow T_e \G$, the vector $c'(0)$ maps to the vector $x:=\alpha e_3+ \beta e_2 \in \mathfrak{m}$.  
Let $\Exp: \mathfrak{g} \rightarrow G$ denote the Lie exponential map.  By Remark \ref{homgeo}, $c(t)=\Theta (\pi (\Exp(tx)))=\Exp(tx)\cdot e.$

Define complex valued functions $f(t)=\cos(tm)+i\frac{\alpha\sin(tm)}{m}$ and $g(t)=i\frac{\beta\sin(tm)}{\epsilon^{1/2}m}.$  For all $t \in \mathbb{R}$, $f(t+T)=f(t)$ and $g(t+T)=g(t)$. Verify that $$\Exp(tx)=(-\alpha(\epsilon^{-1}-1)t,\begin{pmatrix} f(t) & g(t) \\ \\ g(t) & \overline{f(t)}  \end{pmatrix}).$$ Therefore, $$c(t)=\begin{pmatrix} f(t) e^{i\alpha(\epsilon^{-1}-1)t} & g(t) e^{-i\alpha(\epsilon^{-1}-1)t}\\ \\ g(t) e^{i\alpha(\epsilon^{-1}-1)t} & \overline{f(t)} e^{-i\alpha(\epsilon^{-1}-1)t} \end{pmatrix}$$ and 

\begin{eqnarray} c(t+T)& = &\begin{pmatrix} f(t+T) e^{i\alpha(\epsilon^{-1}-1)(t+T)} & g(t+T) e^{-i\alpha(\epsilon^{-1}-1)(t+T)}\\ \\ g(t+T) e^{i\alpha(\epsilon^{-1}-1)(t+T)} & \overline{f(t+T)} e^{-i\alpha(\epsilon^{-1}-1)(t+T)} \end{pmatrix}\nonumber\\\nonumber \\ &=& \begin{pmatrix} f(t) e^{i\alpha(\epsilon^{-1}-1)(t+T)} & g(t) e^{-i\alpha(\epsilon^{-1}-1)(t+T)}\\ \\ g(t) e^{i\alpha(\epsilon^{-1}-1)(t+T)} & \overline{f(t)} e^{-i\alpha(\epsilon^{-1}-1)(t+T)}\end{pmatrix}\nonumber \\ \nonumber \\ &=& c(t) \begin{pmatrix} e^{i\alpha(\epsilon^{-1}-1)T} & 0\\ \\ 0 & e^{-i\alpha(\epsilon^{-1}-1)T} \end{pmatrix} \nonumber \\ \nonumber \\ &=& c(t) \begin{pmatrix} e^{i\epsilon^{-1}S} & 0\\ \\ 0 & e^{-i\epsilon^{-1}S} \end{pmatrix}\nonumber \\ \nonumber \\ &=& \phi^{S}(c(t)).\nonumber 
\end{eqnarray}

\end{proof}

\section{\bf Proof of Theorem \ref{main1}.}

\noindent \textit{Proof of Theorem \ref{main1}.}
Let $M$ be a Riemannian homogeneous three-sphere that is not naturally reductive.  Then there are distinct positive real numbers $x,y,z$ such that $M$ is isometric to a left-invariant metric $g$ on $\G$ admitting an orthonormal left-invariant framing $\{E_1,E_2,E_3\}$ with structure constants as in (\ref{structure2}).  Up to relabeling, we may assume that $z<y<x<0$.  

Define nonzero constants $v_2=\sqrt{\frac{y-z}{x-z}}$ and $v_3=\sqrt{\frac{x-y}{x-z}}$.  Note that $v_2^2+v_3^2=1$ and that 
\begin{equation}\label{key}
v_2^2(x-y)=v_3^2(y-z).
\end{equation}

Define a smooth one-dimensional foliation $\mathcal{F}$ as the orbit foliation of the left-invariant vector field $V=v_2E_2+v_3E_3$.  We will show that $\mathcal{F}$ is a metric foliation that is not homogeneous.

Complete $V$ to an orthonormal framing $\{V,U,W\}$ defined by  \begin{equation}\label{vf}
V=v_2E_2+v_3E_3,\,\,\,\,\,\,\, U=E_1,\,\,\,\,\,\,\,W=-v_3E_2+v_2E_3.
\end{equation}
Use (\ref{use}) to calculate \begin{eqnarray}\label{use3}\nonumber
\,\,\,\,\,\,\,\,\,\,\,\,\,\,\,\,\,\,\, \n_U W=-v_2(x-y+z)E_2-v_3(x-y+z)E_3, && \n_U U=0,  \\  
\n_W U=v_2(-x+y+z)E_2+v_3(x+y-z)E_3, && \n_W W=2v_2v_3(z-x)E_1,\\ \nonumber
 \n_V V=2v_2 v_3(x-z)E_1. \hspace{3.2cm} &&  \\ \nonumber
\end{eqnarray}

Use (\ref{vf}) and (\ref{use3}) to verify that $$g(\n_U U,V)=g(\n_W W,V)=0.$$ Use (\ref{key})-(\ref{use3}) to verify that $g(\n_U W+\n_W U,V)=0.$  
Conclude that $T\mathcal{F}^{\perp}$ is totally geodesic, or equivalently, that $\mathcal{F}$ is a metric foliation.  Let $\omega(\cdot)=g(\n_V V, \cdot)$ be the mean curvature form of $\mathcal{F}$.  Calculate $$d\omega(E_2,E_3)=E_2\omega(E_3)-E_3\omega(E_2)-\omega([E_2,E_3])=-g(\n_V V,2yE_1)=-4v_2v_3y(x-z).$$  Conclude that $\omega$ is not closed and by Lemma \ref{exact} that $\mathcal{F}$ is not homogeneous.
\qed

\section{\bf Proof of Theorem \ref{main2}.}
This section consists of the proof of Theorem \ref{main2}; the proof is presented at the end of the section after a number of preliminary results are derived.  Throughout this section, we let $(M,g)=(\G,h_{\epsilon})$ and let $\mathcal{F}$ denote a one-dimensional metric foliation of $M$. Recall from (\ref{strut2}) that $M$ admits an orthonormal framing $\{Y_1,Y_2,Y_3\}$ with structure constants 
\begin{equation}\label{strut9}
[Y_1,Y_2]=2Y_3\,\,\,\,\,\, [Y_2,Y_3]=2\epsilon^{-1}Y_1\,\,\,\,\,\, [Y_3,Y_1]=2\epsilon^{-1}Y_2.
\end{equation} Moreover, the vector field $Y_3$ is a Killing field.  In particular, its orbits are geodesics in $M$.  By (\ref{use}), the nonzero Christoffel symbols for this framing are given by \begin{eqnarray}\label{use2}\nonumber
&\Gamma_{12}^3=1=-\Gamma_{13}^2,&\\ 
&\Gamma_{23}^1=1=-\Gamma_{21}^3,&\\ \nonumber
 & \hspace{1.4cm} \Gamma_{31}^2=(2\epsilon^{-1}-1)=-\Gamma_{32}^1.
\end{eqnarray}

As $M$ is simply connected, $\mathcal{F}$ is oriented with a globally defined unit length vector field $V$ tangent to $\mathcal{F}$.  Let $\omega(\cdot)=g(\n_V V,\cdot)$ be the mean curvature one form of $\mathcal{F}$.  Our eventual goal is to prove that $\omega$ is closed.  We begin with some preliminary results.

\begin{defn}
Define a subset $\mathcal{O} \subset M$ by $\mathcal{O}=\{p \in M\, \vert\, V \neq \pm Y_3\}.$  
\end{defn}

There exist smooth functions $\psi:\mathcal{O} \rightarrow (0,\pi)$ and $\nu:\mathcal{O} \rightarrow \mathbb{R}/2\pi\mathbb{Z}$ such that \begin{equation}
\label{v}
V=\sin(\psi) \cos(\nu) Y_1+\sin(\psi)\sin(\nu) Y_2+\cos(\psi)  Y_3.
\end{equation} Define vector fields $W$ and $U$ on $\mathcal{O}$ by \begin{equation}
\label{w}
W=\cos(\psi)\cos(\nu)Y_1+\cos(\psi)\sin(\nu) Y_2-\sin(\psi) Y_3,
\end{equation} \begin{equation}
\label{u}
U=-\sin(\nu) Y_1+\cos(\nu) Y_2.
\end{equation} The vector fields $\{V,W,U\}$ constitute an orthonormal framing over $\mathcal{O}$.  

\begin{prop}\label{key2}
On $\mathcal{O}$, $[Y_3,V]=[Y_3,U]=[Y_3,W]=0.$  
\end{prop}

\begin{proof}
We begin by proving that $[Y_3,V]=0$.  Let $\phi^s:M \rightarrow M$, $s \in \mathbb{R}$, denote the isometric flow generated by $Y_3$.  Fix $p \in \mathcal{O}$. For $s\in \mathbb{R}$, let $V_s=V(\phi^{s}(p))$ and $V_s^{\perp}=V^{\perp}(\phi^{s}(p))$.  As $$[Y_3,V](p)=(\mathcal{L}_{Y_3} V)(p)=\lim_{s\rightarrow 0} \frac{d\phi^{-s}(V_s)-V_0}{s},$$ it suffices to prove that $d\phi^{s}(V_0)=V_s$ for all $s$ in some interval about $0$.  As the flow $\phi^s$ is isometric, it suffices to prove that $d\phi^{s}(V_0^{\perp})=V_s^{\perp}$ for all $s$ in a neighborhood of $0$. 

For $\xi \in \mathbb{R}$ close to $0$, define vectors $v_1^{\xi},v_2^{\xi} \in V_0^{\perp}$ by $$v_1^{\xi}=\cos(\xi)U_p+\sin(\xi)W_p, \hspace{.4cm}$$ $$v_2^\xi=-\cos(\xi) U_p+\sin(\xi)W_p.$$  The vectors $v_1^{\xi}$ and $v_2^{\xi}$ make equal angle $\theta(\xi)=\arccos(-\sin(\xi)\sin(\psi_p))$ with $Y_3$.  For $i=1,2$, let $c_i^{\xi}(t)$ denote the geodesic with initial velocity $v_i^{\xi}$.  As $\mathcal{F}$ is a metric foliation, $V^{\perp}$ is a totally-geodesic distribution.  Therefore, the geodesics $c_i^{\xi}$ remain tangent to $V^{\perp}$ for all time.  By Proposition \ref{GEO}, it follows that $$d\phi^{S_{\theta(\xi)}}(V_0^{\perp})=V_{S_{\theta(\xi)}}^{\perp}.$$  

Note that $\theta(\xi)$ carries a neighborhood of $\xi=0$ to a neighborhood of $\theta=\pi/2$.  By formula (\ref{ST}), $S_{\pi/2}=0$ and $\frac{dS_{\theta}}{d\theta}(\pi/2)=2\pi(\epsilon-1)\epsilon^{1/2}$.  As $\epsilon \neq 1$, it follows that $S_{\theta(\xi)}$ carries a neighborhood of $\xi=0$ to a neighborhood of $S=0$, concluding the proof that $[Y_3,V]=0$.

The above proof established that the flow $\phi^{s}$ preserves the distribution $V^{\perp}$.  As it is isometric, it also preserves the distribution $Y_3^{\perp}$. Therefore, $\phi^{s}$ preserves the line field $V^{\perp} \cap Y_3^{\perp}$.  Since $U$ is tangent to this line field and has constant length, $[Y_3,U]=0$.  As $\phi^{s}$ is isometric, the remaining vector field  $W$ in the orthonormal framing $\{V,U,W\}$ is preserved, implying $[Y_3,W]=0$. 
\end{proof}

\begin{lem}\label{equalities}
The following hold on $\mathcal{O}$.

\begin{enumerate}
\item $Y_3(\psi)=0$.
\item $Y_3(\nu)+2\epsilon^{-1}=0$.
\item $W(\psi)=0$.
\item $U(\psi)+\sin(\psi)W(\nu)+(2-2\epsilon^{-1})\sin^2(\psi)=0.$
\item $V(\psi)=0.$
\item $g(U,[U,V])=0.$
\item $V(U(\psi))=0.$
\item $Y_3(V(\nu))=0.$
\item The set $\{p \in \mathcal{O}\, \vert\, \psi(p)=\pi/2\}$ has empty interior.
\item $V(V(\nu))=0.$
\end{enumerate}
\end{lem}

\noindent \textit{Proof of (1) and (2).} By Proposition \ref{key2}, $[Y_3,V]=0$.  Substitute (\ref{v}) into this equality and simplify using (\ref{strut9}) and (\ref{use2}).\\

\noindent \textit{Proof  of (3):}  As $V^{\perp}$ is totally geodesic, $g(\n_W W, V)=0$.  Substitute (\ref{v}) and (\ref{w}) into this equality and simplify using (\ref{strut9}) and (\ref{use2}).\\

\noindent \textit{Proof of (4):} As $V^{\perp}$ is totally geodesic, $g(\n_U W+\n_W U,V)=0$.  Substitute (\ref{v})-(\ref{u}) into this equality and simplify using (\ref{strut9}) and (\ref{use2}).\\

\noindent \textit{Proof of (5):} Use $Y_3=\cos(\psi)V-\sin(\psi)W$ and (1) and (3) to deduce $\cos(\psi)V(\psi)=0$.  Conclude that (5) holds when $\cos(\psi)\neq 0$.  Apply the derivation $V$ to the last equality to deduce $0=V(\cos(\psi)V(\psi))=-\sin(\psi)V(\psi)^2+\cos(\psi)V(V(\psi)).$  Conclude that (5) also holds when $\cos(\psi)=0$.\\

\noindent \textit{Proof of (6):} As $V^{\perp}$ is totally geodesic and the framing $\{V,W,U\}$ is orthonormal, $0=g(-\n_U U,V)=g(U,\n_U V)=g(U,\n_U V)-g(U, \n_V U)=g(U, [U,V])$.\\

\noindent \textit{Proof of (7):} By (6), $[U,V]$ lies in the span of $W$ and $V$. By (3) and (5), $W$ and $V$ annihilate $\psi$. Hence, $[U,V]$ annihilates $\psi$.  Therefore, $0=[U,V](\psi)=U(V(\psi))-V(U(\psi))=-V(U(\psi))$.\\

\noindent \textit{Proof of (8):} By Proposition \ref{key2}, $[Y_3,V]=0$. By (2), $Y_3(\nu)=-2\epsilon^{-1}$.  Therefore $Y_3(V(\nu))=V(Y_3(\nu))=V(-2\epsilon^{-1})=0$.\\

\noindent \textit{Proof of (9):} Suppose to the contrary that there is an open ball $B \subset \mathcal{O}$ on which $\psi \equiv \pi/2$.  By (\ref{w}), $W=-Y_3$ on $B$. By equalities (4) and (2), the following absurdity holds on $B$: $$0=U(\psi)+\sin(\psi)W(\nu)+(2-2\epsilon^{-1})\sin^2(\psi)=W(\nu)+(2-2\epsilon^{-1})=2.$$

\noindent \textit{Proof of (10):} By equality (2), $0=V(Y_3(\nu))$.  As $Y_3=\cos(\psi)V-\sin(\psi)W,$ $$0=V([\cos(\psi)V(\nu)-\sin(\psi)W(\nu)]).$$  Solve for $\sin(\psi)W(\nu)$ in (4) and substitute into the equality above to derive $$0=V([\cos(\psi)V(\nu)+U(\psi)+(2-2\epsilon^{-1})\sin^2(\psi)]).$$  Use equalities (5) and (7), to conclude that $\cos(\psi)V(V(\nu))=0$.  Equality (10) now follows from (9).\\ \qed\\

Define $f: \mathcal{O} \rightarrow \mathbb{R}$ by $f=V(\nu)\sin(\psi)+(\epsilon^{-1}-1)\sin(2\psi).$ Use (\ref{use2}), (\ref{v}), (\ref{u}), and Lemma \ref{equalities}-(5) to derive $\n_{V} V = fU.$   Therefore, $$\omega(\cdot)=g(fU,\cdot)$$ on $\mathcal{O}.$  We conclude this section with the proof of Theorem \ref{main2}.\\

\noindent \textit{Proof of Theorem \ref{main2}.} By Lemma \ref{exact}, we must show that $d\omega=0$ on $M$. It suffices to prove that $d\omega=0$ on $\mathcal{O}$ since by (\ref{use2}), $\omega=0$ on the interior of $M \setminus \mathcal{O}$.  

By Lemma \ref{equalities}-(6), $$d\omega(V,U)=Vg(fU,U)-Ug(fU,V)-g(fU,[V,U])=V(f).$$  Use Lemma \ref{equalities}-(5,10) to conclude that $d\omega(V,U)=0.$  By Proposition \ref{key2}, $$d\omega(Y_3,U)=Y_3g(fU,U)-Ug(fU,Y_3)-g(fU,[Y_3,U])=Y_3(f).$$ Use Lemma \ref{equalities}-(1,8) to conclude that $d\omega(Y_3,U)=0.$  By Proposition \ref{key2}, $$d\omega(Y_3,V)=Y_3g(fU,V)-Vg(fU,Y_3)-g(fU,[Y_3,V])=0.$$

As the vector fields $\{V,U,Y_3\}$ are linearly independent over $\mathcal{O}$, $d\omega=0$ on $\mathcal{O}$, concluding the proof.  
\qed

\end{document}